\documentclass[a4paper]{article}
\usepackage{a4wide}

\usepackage[utf8]{inputenc}
\usepackage{subfigure}

\usepackage[sc]{mathpazo} 

\usepackage{graphicx}

\usepackage{amsmath}
\usepackage{amsfonts}
\usepackage{amssymb}
\usepackage{amsthm}

\newtheorem{theorem}{Theorem}

\newtheorem{res}[theorem]{Proposition}

\newcommand{\C}{\mathcal{C}}

\newcommand{\Flow}{\varphi}

\newcommand{\HHH}{\mathcal{H}}
\renewcommand{\H}{\mathsf{H}}

\newcommand{\KK}{\mathbb{K}}
\renewcommand{\L}{\mathsf{L}}

\newcommand{\NN}{\mathbb{N}}

\newcommand{\RR}{\mathbb{R}}
\newcommand{\Tension}{\theta}

\newcommand{\ZZ}{\mathbb{Z}}

\newcommand{\comb}{\mathsf{comb}}

\newcommand{\conv}{\mathsf{conv}}

\newcommand{\defn}[1]{{\bf #1}}

\renewcommand{\dim}{\mathsf{dim}\:}

\newcommand{\head}{\mathsf{head}}

\newcommand{\ideal}[1]{\langle #1 \rangle}

\newcommand{\lto}{\prec} 
\newcommand{\mFlow}{\bar{\Flow}}
\newcommand{\mTension}{\bar{\Tension}}

\newcommand{\onenorm}[1]{|| #1 ||_1}

\newcommand{\pull}{\mathsf{pull}}

\newcommand{\rar}[0]{\rightarrow}

\newcommand{\relint}{\mathsf{relint}\:}

\newcommand{\supp}{\mathsf{supp}}

\newcommand{\tail}{\mathsf{tail}}

\renewcommand{\vert}{\mathsf{vert}}

\author{Felix Breuer\thanks{Freie Universit\"at Berlin, Arnimallee 3, 14195 Berlin, Germany, \textsf{felix.breuer@fu-berlin.de}. Supported by the DFG research training group ``Methods for Discrete Structures'' (GrK 1408).}
\and Aaron Dall \thanks{\textsf{adall1979@gmail.com}}}
\title{\textit{\textbf{Viewing counting polynomials as Hilbert functions via Ehrhart theory}}}

\begin{document}
\maketitle
\begin{abstract}
Steingr\'{\i}msson (2001) showed that the chromatic polynomial of a graph is the Hilbert function of a relative Stanley-Reisner ideal. We approach this result from the point of view of Ehrhart theory and give a sufficient criterion for when the Ehrhart polynomial of a given relative polytopal complex is a Hilbert function in Steingr\'{\i}msson's sense. We use this result to establish that the modular and integral flow and tension polynomials of a graph are Hilbert functions.


\end{abstract}


\section{Introduction}
\label{sec:in}

Steingr\'{\i}msson \cite{Steingrimsson01} showed that the proper $k+1$-colorings of a graph $G$ are in bijection with the  monomials of degree $k$ in a polynomial ring $\KK[x_1,\ldots,x_n]$ that lie inside a square-free monomial ideal $I_2$, but outside a square-free monomial ideal $I_1$. In other words, he showed that the chromatic polynomial $\chi_G$ of $G$ is the Hilbert function of a relative Stanley-Reisner ideal. To this end, he used a clever combinatorial construction to describe the ideals $I_1$ and $I_2$ explicitly. 

In this article we approach the problem from the point of view of Ehrhart theory, which allows us to arrive quickly at a sufficient criterion  for when the Ehrhart polynomial of a given relative polytopal complex is a Hilbert function in Steingr\'{\i}msson's sense: 
\begin{quotation}
\noindent
The Ehrhart function of a relative polytopal complex in which all faces are compressed is the Hilbert function of a relative Stanley-Reisner ideal.
\end{quotation}
See Theorem~\ref{thm:relative-ehrhart-hilbert}. We then apply this general result to establish that four other counting polynomials defined in terms of graphs are Hilbert functions: the modular flow and tension polynomials and their integral variants. Also, we are able to improve Steingr\'{\i}msson's result insofar as we are able to obtain the chromatic polynomial $\chi_G(k)$ itself as a Hilbert function, and not only the shifted chromatic polynomial $\chi_G(k+1)$. We conclude the paper by a giving another more algebraic proof of our geometric theorem, which allows us to generalize the result further.

These results have been developed in the authors' respective theses \cite{Dall08} and \cite{Breuer09}, to which we refer the interested reader for further details and additional material.

This paper is organized as follows. After some preliminary definitions in Section~\ref{sec:definitions} we review Steingr\'{\i}msson's theorem and related work in Section~\ref{sec:steingrimsson}. In Section~\ref{sec:hilbert-vs-ehrhart} our main result is derived. In Section~\ref{sec:hil:applications} we apply this result to show that all five counting polynomials are Hilbert functions. We present a generalization of our main result in Section~\ref{sec:hil:non-square-free-non-standard} along with a more algebraic proof. Finally, we give some constraints on the coefficients of the polynomials and discuss questions for further research in Section~\ref{sec:questions}.


\section{Preliminary Definitions}
\label{sec:definitions}

Before we begin, we gather some definitions. We recommend the textbooks \cite{BeckRobins07}, \cite{MillerSturmfels05}, \cite{Stanley96}, \cite{Eisenbud94}, \cite{Schrijver86} and \cite{West01} as references.

The \defn{Ehrhart function} $\L_A$ of any set $A\subset \RR^n$ is defined by $\L_A(k)=|\ZZ^n\cap k\cdot A|$ for $k\in\NN$. A \defn{lattice polytope} is a polytope in $\RR^n$, such that all vertices are integer points. It is a theorem of Ehrhart that the Ehrhart function $\L_P(k)$ of a lattice polytope is a polynomial in $k$. Two polytopes $P,Q$ are \defn{lattice isomorphic}, $P\approx Q$, if there exists an affine isomorphism $A$ such that $A|_{\ZZ^n}$ is a bijection onto $\ZZ^n$ and $AP=Q$. A \defn{$d$-simplex} is the convex hull of  $d+1$ affinely independent points. A $d$-simplex is \defn{unimodular} if it is lattice isomorphic to the convex hull of $d+1$ standard unit vectors. A lattice polytope is \defn{empty} if the only lattice points it contains are its vertices. A \defn{hyperplane arrangement} is a finite collection $\HHH$ of affine hyperplanes and $\bigcup\HHH$ denotes the union of all of these.

A \defn{polytopal complex} is a finite collection $\C$ of polytopes in some $\RR^n$ with the following two properties: If $P\in\C$ and $F$ is a face of $P$, then $F\in\C$; and if $P,Q\in\C$ then $F=P\cap Q\in \C$ and $F$ is common face of both $P$ and $Q$. The polytopes in $\C$ are also called \defn{faces} and $\bigcup\C$ denotes the union of all faces of $\C$. A (geometric) \defn{simplicial complex} is a polytopal complex in which all faces are simplices. An \defn{abstract simplicial complex} is a set $\Delta$ of subsets of a finite set $V$, such that $\Delta$ is closed under taking subsets. A geometric simplicial complex $\Delta$ gives rise to an abstract simplicial complex $\comb(\Delta)$ via $\comb(\Delta)=\{ \sigma | \text{$\sigma$ is the vertex set of some $F\in\C$}\}$. A polytopal complex $\C'$ that is a subset $\C'\subset\C$ of a polytopal complex $\C$ is called a \defn{subcomplex} of $\C$. Subcomplexes of abstract simplical complexes are defined similarly. Given a collection $S$ of polytopes in $\RR^n$ such that for any $P,Q\in S$ the set $P\cap Q$ is a face of both $P$ and $Q$, the polytopal complex $\C$ \defn{generated} by $S$, is $\C=\{F | \text{$F$ a face of $P\in S$}\}$. A \defn{subdivision} of a polytopal complex $\C$ is a polytopal complex $\C'$ such that $\bigcup\C=\bigcup\C'$ and every face of $\C'$ is contained in a face of $\C$. A \defn{triangulation} is a subdivision in which all faces are simplicies. A \defn{unimodular triangulation} is a triangulation in which all simplices are unimodular.

Let $\KK[x]=\KK[x_1,\ldots,x_n]$ denote the polynomial ring in $n$ variables over some field $\KK$ equipped with the standard grading by degree $\KK[x]=\bigoplus_{k\geq0}R_k$, where $R_k$ is the $\KK$-vector space generated by all monomials of degree $k$. A \defn{graded $\KK[x]$-module} is a module $M$ that can be written as a direct sum of abelian groups $M=\bigoplus_{-\infty}^\infty M_k$ such that $R_iM_j\subset M_{i+j}$ for all $i$ and $j$. The \defn{Hilbert function} $\H_M$ of $M$ is defined by $\H_M(k)=\dim_\KK M_k$. Let $I_1$ be a monomial ideal in $\KK[x]$ and consider the quotient ring $\KK[x]/I_1$ graded by degree. Then $\H_{\KK[x]/I_1}(k)=|\{x^a\in\KK[x] | x^a\not\in I_1, \deg(x^a)=k\}|$. Furthermore let $I_2\supset I_1$ be another monomial ideal. By abuse of notation we also denote by $I_2$ the ideal in $\KK[x]/I_1$ generated by the same set of monomials. We can write $I_2=\bigoplus_{k\geq 0} I_2^k$ where $I_2^k$ is the vector space generated by the monomials $x^a$ in $I_2$ with $\deg(x^a)=k$ that are non-zero in $\KK[x]/I_1$. Then $R_i I_2^k\subset I_2^{i+k}$, where the product is taken in $\KK[x]/I_1$, and $\H_{I_2}(k)=|\{x^a\in I_2\setminus I_1|\deg(x^a)=k\}|$. A \defn{term order} on $\KK[x]$ is a total order on the monomials $x^a\in\KK[a]$ such that $1\lto x^a$ for all $a\in\ZZ^n_{>0}$ and $x^a\lto x^b$ implies $x^{a+c}\lto x^{b+c}$ for all $a,b,c\in\ZZ^n_{\geq0}$.

We consider oriented graphs that may have loops and multiple edges. Note, however, that the values of the five counting polynomials do not depend on the orientation of the graph and are thus invariants of the underlying unoriented graph. Formally, a \defn{graph} is a tuple $(V,E,\head,\tail)$, where $V$ is a finite vertex set, $E$ is a finite edge set and $\head:E\rar V$ and $\tail:E\rar V$ are maps. Graph theoretic concepts such as adjacency, paths, connectivity, etc.\ are defined in the usual way. We note that a cycle in the underlying unoriented graph can be coded as a map $c:E\rar\{0,\pm1\}$ where $c_e=+1$ if the direction in which $e$ is traversed is consistent with the orientation of $e$ in $G$, $c_e=-1$ if the direction of traversal is opposite to the orientation in $G$ and $c_e=0$ if $e$ does not lie on the cycle. Here we view $c$ both as a map and as a vector as we shall do with all maps defined in this article.

Let $k\in\ZZ_{>0}$. A \defn{$k$-coloring} of $G$ is a map $x:V\rar \{0,\ldots,k-1\}$ and it is called \emph{proper} if $x_v\not=x_u$ whenever $u\sim v$. The \defn{chromatic polynomial $\chi_G$} is defined such that $\chi_G(k)$ is the number of proper $k$-colorings of $G$.

A \defn{$k$-tension} of $G$ is a map $t:E\rar\{-k+1,\ldots,k-1\}$ such that 
\begin{eqnarray}
\label{eqn:conservation-of-tension}
\sum_{e\in E} c_e t_e= 0 && \text{ for every cycle $c$ in $G$.}
\end{eqnarray} 
Similarly, a \defn{$\ZZ_k$-tension} of $G$ is a map $t:E\rar\ZZ_k$ such that (\ref{eqn:conservation-of-tension}) holds in $\ZZ_k$. A tension is \defn{nowhere zero} if $t(e)\not=0$ for all $e\in E$. Now we define functions $\Tension_G$ and $\mTension _G$ as follows: $\Tension_G(k)$ is the number of nowhere zero $k$-tensions of $G$ and $\mTension_G(k)$ is the number of nowhere zero $\ZZ_k$-tensions of $G$. Both $\Tension_G(k)$ and $\mTension_G(k)$ are polynomials in $k$, called the \defn{integral} and the \defn{modular tension polynomial}, respectively.

A \defn{$k$-flow} of $G$ is a map $f:E\rar\{-k+1,\ldots,k-1\}$ such that 
\begin{eqnarray}
\label{eqn:conservation-of-flow}
\sum_{\substack{e\in E\\ \head(e)=v}} f_e - \sum_{\substack{e\in E\\ \tail(e)=v}} f_e= 0 && \text{ for every vertex $v$ of $G$.}
\end{eqnarray}
 Similarly, a \defn{$\ZZ_k$-flow} of $G$ is a map $f:E\rar\ZZ_k$ such that (\ref{eqn:conservation-of-flow}) holds in $\ZZ_k$. The functions $\Flow_G$ and $\mFlow_G$ are defined as follows: $\Flow_G(k)$ is the number of nowhere zero $k$-flows of $G$ and $\mFlow_G(k)$ is the number of nowhere zero $\ZZ_k$-flows of $G$. Both $\Flow_G(k)$ and $\mFlow_G(k)$ are polynomials in $k$, called the \defn{integral} and the \defn{modular flow polynomial}, respectively. More about these polynomials can be found in \cite{BeckZaslavsky06a}, \cite{BeckZaslavsky06b}, \cite{Kochol02} and \cite{Breuer09}.

\section{Steingr\'{\i}msson's theorem and related work}
\label{sec:steingrimsson}

Steingr\'{\i}msson \cite{Steingrimsson01} showed that for any graph $G$ the chromatic polynomial $\chi_G(k+1)$ shifted by one is the Hilbert function of a module with a particular structure.

\begin{theorem}{(Steingr\'{\i}msson \cite[Theorem~9]{Steingrimsson01}) }
\label{thm:hil:steingrimsson}
For any graph $G$, there exists a number $n$, a square-free monomial ideal $I_1$ in the polynomial ring over $n$ variables $\KK[x]=\KK[x_1,\ldots,x_n]$ and a square-free monomial ideal $I_2$ in $\KK[X]/I_1$ such that
\[
\H_{I_2}(k) = \chi_G(k+1)
\]
for all $k\in \ZZ_{>0}$, where $\H_{I_2}$ denotes the Hilbert function of $I_2$ with respect to the standard grading and $\chi_G$ denotes the chromatic polynomial of $G$.
\end{theorem}

In \cite{Steingrimsson01} Steingr\'{\i}msson went on to define the coloring complex of a graph to be the simplicial complex given by the square-free monomial ideal $I_2$. In the case of colorings the ideal $I_1$ has a simple description, so bounds on the $f$-vector of the coloring complex translate into bounds on the coefficients of the chromatic polynomial. The articles \cite{Jonsson05},\cite{Hultman07},\cite{HershSwartz08}, building on Steingr\'{\i}msson's work, have mainly dealt with showing various properties of the coloring complex. Steingr\'{\i}msson himself gave a combinatorial description of the coloring complex and determined its Euler characteristic to be the number of acyclic orientations of $G$. To some extent this was already known: Welker observed that the coloring complex of a graph $G=(V,E)$ is the same as a complex appearing in the article \cite{HerzogReinerWelker98} by Herzog, Reiner and Welker, where this complex is shown to be homotopy equivalent to a wedge of spheres of dimension $|V|-3$ and the number of spheres is the number of acyclic orientations of $G$ minus one. Jonsson \cite{Jonsson05} showed the coloring complex to be constructible and hence Cohen-Macaulay. This result was improved by Hultman \cite{Hultman07} who showed the coloring complex to be shellable and by Hersh and Swartz \cite{HershSwartz08} who showed that the coloring complex has a convex ear decomposition. These results translate into bounds on the coefficients of $\chi_G$.

In this article we concentrate on establishing the structural result that the four counting polynomials are Hilbert functions in Steingr\'{\i}msson's sense. We do not focus on the task of obtaining bounds on the coefficients, but we make some general observations and suggest directions for future research.

\section{Hilbert equals Ehrhart}
\label{sec:hilbert-vs-ehrhart}

In this section we relate Ehrhart functions of certain complexes to Hilbert functions of ideals defined in terms of these complexes. We begin with the well-known relation between simplicial complexes and the corresponding Stanley-Reisner ideals, move on to relative simplicial complexes and relative Stanley-Reisner ideals before we finally consider relative polytopal complexes.\footnote{We introduce the polytopal Stanley-Reisner ideals corresponding to polytopal complexes only later in Section~\ref{sec:hil:non-square-free-non-standard}.} As Ehrhart functions are defined in terms of geometric simplicial complexes while Stanley-Reisner ideals are defined in terms of abstract simplicial complexes, all the complexes we consider live in both worlds. A geometric simplicial complex $\Delta$ has an abstract simplicial complex $\comb(\Delta)$ associated with it, see Section~\ref{sec:definitions}.

Let $\Delta$ be an abstract simplicial complex on the ground set $V$. We identify the elements of the ground set of $\Delta$ with the variables in the polynomial ring $\KK[x_v:v\in V]=:\KK[x]$. Thus sets $S\subset V$ correspond to square-free monomials in $\KK[x]$. The \defn{Stanley-Reisner ideal} $I_\Delta$ of $\Delta$ is generated by the monomials corresponding to the minimal non-faces of $\Delta$, more precisely
\[
I_\Delta := \langle x^u\in K[x] | \supp(u)\not\in\Delta \rangle.
\]
Then, the \defn{Stanley-Reisner ring} of $\Delta$ is the quotient $\KK[\Delta]=\KK[x]/I_\Delta$. We equip the ring $\KK[x]$ with the standard grading, that is for any monomial $x^u\in\KK[x]$ we have $\deg(x^u)=\onenorm{u}=\sum_{i=1}^n u_i$. The fundamental result about Stanley-Reisner rings is this:

\begin{theorem}{\cite{Stanley96} }
\label{thm:hilbert-function-and-f-vector}
Let $\Delta$ be a $d$-dimensional (abstract) simplicial complex with $f_i$ faces of dimension $i$ for $0\leq i\leq d$. Then the Hilbert function $\H_{\KK[\Delta]}$ of the Stanley-Reisner ring $\KK[\Delta]$ satisfies
\begin{eqnarray}
\label{eqn:hilbert-function-and-f-vector}
\H_{\KK[\Delta]}(k) & = & \sum_{i=0}^{d}f_i {k-1 \choose i}
\end{eqnarray}
for $k\in\ZZ_{>0}$ and $\H_{\KK[\Delta]}(0)=1$.
\end{theorem}
We remark that the right-hand side of (\ref{eqn:hilbert-function-and-f-vector}) evaluated at zero gives $\sum_{i=0}^{d}f_i {-1 \choose i}=\chi(\Delta)$, the Euler characteristic of $\Delta$.

If we are given a geometric simplicial complex $\Delta$ we will generally use $\vert(\Delta)$ as the ground set of the abstract simplicial complex $\comb(\Delta)$ and identify the variables of $\KK[x]$ with the vertices of $\Delta$. In this case we use $I_\Delta$ to refer to $I_{\comb(\Delta)}$ and similarly for $\KK[\Delta]$. 

Now the Ehrhart functions of a unimodular $d$-dimensional lattice simplex $\sigma^d$ and its relative interior $\relint \sigma^d$ are, respectively,
\begin{eqnarray}
\label{eqn:ehrhart-of-closed-simplex}\label{eqn:ehrhart-of-open-simplex}
\L_{\sigma^d}(k) = {k+d \choose d}
& \text{ and } &
\L_{\relint \sigma^d}(k) = {k-1 \choose d}.
\end{eqnarray}
Taken together, (\ref{eqn:hilbert-function-and-f-vector}) and (\ref{eqn:ehrhart-of-open-simplex}) tell us that for any (geometric) simplicial complex $\Delta$ in which all simplices are unimodular, the Ehrhart function $\L_\Delta(k) = |\ZZ^d\cap k\bigcup\Delta|$ of $\Delta$ satisfies
\begin{eqnarray}
\L_\Delta(k) \;\; = & \sum_{\sigma\in\Delta} \L_{\relint \sigma}(k) \;\; = \;\; \sum_{i=0}^{d}f_i {k-1 \choose i} & = \;\; \H_{\KK[\Delta]}(k)
\end{eqnarray}
for all $k\in\ZZ_{>0}$. Simply put: the Ehrhart function of a unimodular geometric simplicial complex and the Hilbert function of the corresponding Stanley-Reisner ring coincide. This fact is well-known, see for example \cite{MillerSturmfels05}. Taking the above approach and calculating the Ehrhart functions of open simplices, however, allows us to do without M\"obius inversion.

For our purpose we need a more general concept than that of a Stanley-Reisner ring. For an abstract simplicial complex $\Delta$ the Hilbert function $\H_{\KK[\Delta]}(k)$ counts all those monomials $x^u$ of degree $k$ with $\supp(u)\in\Delta$. We are interested in a pair of simplicial complexes $\Delta'\subset\Delta$, the former being a subcomplex of the latter, and want to count those monomials $x^u$ such that $\supp(u)\not\in\Delta'$ but $\supp(u)\in\Delta$. To that end we follow Stanley \cite{Stanley96} in calling a pair of simplicial complexes $\Delta'\subset\Delta$ a \defn{relative simplicial complex}. We denote by $I_{\Delta/\Delta'}$ the ideal in $\KK[\Delta]$ generated by all monomials $x^u$ with $\supp(u)\not\in\Delta'$. We call this the \defn{relative Stanley-Reisner ideal}. Its Hilbert function $\H_{I_{\Delta/\Delta'}}(k)$ counts the number of non-zero monomials $x^u$ of degree $k$ in $I_{\Delta'}\setminus I_{\Delta}$ or, equivalently, the number of non-zero monomials $x^u$ in $\KK[x]$ with $\supp(u)\in\Delta\setminus\Delta'$. (Notice how the roles of $\Delta$ and $\Delta'$ swap, depending on whether we formulate the condition using ideals or using complexes). Now, as Stanley remarks, Theorem~\ref{thm:hilbert-function-and-f-vector} carries over to the relative case.

\begin{theorem}{\cite{Stanley96}}
\label{thm:relative-hilbert-function-and-f-vector}
Let $\Delta'\subset\Delta$ be a relative d-dimensional abstract simplicial complex and let $f_i$ denote the number of $i$-dimensional simplices in $\Delta\setminus\Delta'$. Then for all $k\in\ZZ_{>0}$
\begin{eqnarray}
\label{eqn:relative-hilbert-function-and-f-vector}
\H_{I_{\Delta/\Delta'}}(k) & = & \sum_{i=0}^{d}f_i {k-1 \choose i}.
\end{eqnarray}
\end{theorem}

If $\Delta$ is a geometric simplicial complex and $\Delta'\subset \Delta$ a subcomplex, we also call the pair $\Delta'\subset \Delta$ a relative geometric simplicial complex and define its relative Stanley-Reisner ideal $I_{\Delta/\Delta'}$ to be $I_{\comb(\Delta)/\comb(\Delta')}$.

By the same argument as above, we conclude that for any relative $d$-di\-men\-sion\-al geometric simplicial complex $\Delta'\subset\Delta$, all faces of which are unimodular,
\begin{eqnarray}
\L_{\bigcup\Delta\setminus\bigcup\Delta'}(k) = \sum_{\sigma\in\Delta\setminus\Delta'} \L_{\relint(\sigma)}(k) = \sum_{i=0}^{d}f_i {k-1 \choose i} =  \H_{I_{\Delta/\Delta'}}(k)
\end{eqnarray}
for all $k\in\ZZ_{>0}$, i.e.\ the Ehrhart function of a relative simplicial complex with unimodular faces and the Hilbert function of the associated relative Stanley-Reisner ideal coincide. Moreover this function is a polynomial in $k$ as 
\[
{k-1 \choose i} = \frac{1}{i!}\prod_{i=1}^i (k-i)
\]
is a polynomial for every $i\in\ZZ_{\geq0}$ using the convention that $i!=\prod_{j=1}^ij$ and empty products are 1.

To be able to deal with the applications in Section~\ref{sec:hil:applications} we need to go one step further. The complexes we will be dealing with, are not going to be simplicial. Their faces will be polytopes. So we define a \defn{relative polytopal complex} to be a pair $\C'\subset\C$ of polytopal complexes, the former a subcomplex of the latter. Our goal is to realize the Ehrhart function $\L_{\bigcup\C\setminus\bigcup\C'}(k)$ as the Hilbert function of a relative Stanley-Reisner ideal.

By the above arguments, it would suffice to require that $\C$ has a unimodular triangulation. But for the sake of convenience we would like to impose a condition on $\C$ that can be checked one face at a time. Requiring that each face of $\C$ has a unimodular triangulation would not be sufficient. A unimodular triangulation $\Delta_F$ for each face $F\in\C$ does not guarantee that $\bigcup_{F\in\C}\Delta_F$ is a unimodular triangulation of $\C$: It may be that for faces $F_1,F_2\in\C$ that share a common face $F=F_1\cap F_2$ the unimodular triangulations $\Delta_{F_1}$ and $\Delta_{F_2}$ do not agree on $F$, i.e.\ \[\{F\cap f_1 \:|\: f_1\in\Delta_{F_1}\}\not=\{F\cap f_2 \:|\: f_2\in\Delta_{F_2}\}.\] Fortunately there is the notion of a compressed polytope: it suffices to require of each face $F\in\C$ individually that $F$ is compressed, to guarantee that $\C$ as a whole has a unimodular triangulation.

Let $P\subset\RR^d$ be a lattice polytope. Let $\prec$ be a total ordering of the lattice points in $P$. The \defn{pulling triangulation $\pull(P;\prec)$} of $P$ with respect to the total ordering $\prec$ is defined recursively as follows.  If $P$ is an empty simplex, then $\pull(P;\prec)$ is the complex generated by $P$. Otherwise $\pull(P;\prec)$ is the complex generated by the set of polytopes
\begin{eqnarray*}
\bigcup_F \left\{ \conv\{v,G\} : G\in \pull(F;\prec) \right\}
\end{eqnarray*}
where $v$ is the $\prec$-minimal lattice point in $P$ and the union runs over all faces $F$ of $P$ that do not contain $v$. See also Sturmfels \cite{Sturmfels96}. This construction yields a triangulation and the vertices of $\pull(P,\prec)$ are lattice points in $P$. Pulling triangulations need not be unimodular, in fact the simplices in a pulling triangulation do not even have to be empty! Polytopes whose pulling triangulation is always unimodular get a special name. A polytope $P$ is \defn{compressed} if for any total ordering $\prec$ on the vertex set the pulling triangulation $\pull(P,\prec)$ is unimodular. These definitions have the following well-known properties.

\begin{res}
\label{res:hil:compressed-polytopes}
\begin{enumerate}
\item Any $\dim(P)$-dimensional simplex in $\pull(P,\prec)$ contains the $\prec$-minimal lattice point $v$ in $P$ as a vertex.
\item $\pull(F,\prec)=\pull(P,\prec)\cap F$ for any total order $\prec$ on the lattice points in $P$ and any face $F$ of $P$.
\item All faces of a compressed polytopes are compressed.
\item Compressed polytopes are empty.
\end{enumerate}
\end{res}

A proof of this proposition can be found in \cite{Breuer09}. For more information on pulling triangulations and compressed polytopes we refer to \cite{Sturmfels96}, \cite{OhsugiHibi01}, \cite{Sullivant04}, \cite{Sturmfels91} and \cite{DeLoeraRambauSantos09}.

Now, if $\C$ is a polytopal complex with integral vertices such that every face $P\in\C$ is compressed, then we can fix an arbitrary total order $\prec$ on $\bigcup\C\cap\ZZ^d$ and construct the pulling triangulations $\pull(P,\prec)$ of all faces $P\in\C$ with respect to that one global order $\prec$. By Proposition~ \ref{res:hil:compressed-polytopes} this means that the for any two $P_1,P_2\in\C$ that share a common face $F=P_1\cap P_2$ the triangulations induced on $F$ agree: $\pull(P_1,\prec)\cap F=\pull(F,\prec)=\pull(P_2,\prec)\cap F$. Thus $\pull(\C,\prec):=\bigcup_{F\in\C} \pull(F,\prec)$ is a unimodular triangulation of $\C$ with $\vert(\C)=\vert(\pull(\C,\prec))$. We abbreviate $\Delta:=\pull(\C,\prec)$.

If $\C'$ is any subcomplex of $\C$, we define $\Delta'$ to be the subcomplex of $\Delta$ consisting of those faces $F\in\Delta$ such that $F\subset\bigcup\C'$. So
\begin{eqnarray}
\L_{\bigcup\C\setminus\bigcup\C'}(k)=\L_{\bigcup\Delta\setminus\bigcup\Delta'}(k)=\H_{I_{\Delta/\Delta'}}(k)
\end{eqnarray}
for $k\in\ZZ_{>0}$ which means that we have realized the Ehrhart function of $\bigcup\C\setminus\bigcup\C'$ as the Hilbert function of the relative Stanley-Reisner ideal $I_{\Delta/\Delta'}$. Moreover, we have already seen that this function is a polynomial. We summarize these results in the following theorem.

\begin{theorem}
\label{thm:relative-ehrhart-hilbert}
Let $\C$ be a polytopal complex. If all faces of $\C$ are compressed lattice polytopes, then for any subcomplex $\C'\subset \C$ there exists a relative Stanley-Reisner ideal $I_{\Delta,\Delta'}$ such that for all $k\in\ZZ_{>0}$
\[
\L_{\bigcup\C\setminus\bigcup\C'}(k) = \H_{I_{\Delta/\Delta'}}(k)
\]
and this function is a polynomial.
\end{theorem}


\section{Counting Polynomials as Hilbert Functions}
\label{sec:hil:applications}

In this section we apply Theorem~\ref{thm:relative-ehrhart-hilbert} to obtain analogues of Steingr\'{\i}msson's Theorem \ref{thm:hil:steingrimsson} for all five counting polynomials. A useful tool in this context is going to be the following theorem by Ohsugi and Hibi which states that lattice polytopes that are slices of the unit cube are automatically compressed.

\begin{theorem}{(Ohsugi and Hibi \cite{OhsugiHibi01}) }
\label{res:hil:slices-of-cube-compressed}
Let $P$ be a lattice polytope in $\RR^n$. If $P$ is lattice isomorphic to the intersection of an affine subspace with the unit cube, i.e.\ $P\approx [0,1]^n\cap L$ for some affine subspace $L$, then $P$ is compressed.\footnote{Actually, Ohsugi and Hibi showed a more general result, but this will suffice for our purposes. Interestingly, Sullivant \cite{Sullivant04} noted that this condition is also necessary for a lattice polytope to be compressed.}
\end{theorem}


\subsection*{Integral Flow and Tension Polynomials}

Let $S$ be the linear subspace of $\RR^E$ given by (\ref{eqn:conservation-of-flow}). The $k$-flows of $G$ are in bijection with the lattice points in $k\cdot (-1,1)^E \cap S$. Furthermore let $\HHH=\{ \{x| x_e=0\} \: | \: e\in E\}$ denote the arrangement of all coordinate hyperplanes. Then the nowhere zero $k$-flows of $G$ are in bijection with the lattice points in $k\cdot (-1,1)^E \cap S\setminus \bigcup \HHH$. The closures of the components of $(-1,1)^E \cap S\setminus \bigcup \HHH$ are of the form $(\prod_{e\in E} [a_e,a_e+1]) \cap S$ for some $a\in \{-1,0\}^E$. Let $\C$ be the polytopal complex generated by these and let $\C'$ be the subcomplex of all faces of $\C$ contained in the boundary of $[-1,1]^E$ or contained in one of the coordinate hyperplanes. Then  $\Flow_G(k)=\L_{\bigcup \C\setminus\bigcup \C'}(k)$. Because (\ref{eqn:conservation-of-flow}) gives rise to a totally unimodular matrix, the maximal faces of $\C$ are lattice polytopes.\footnote{We refer to \cite{Schrijver86} for the concept of a totally unimodular matrix and related results.} Moreover by Theorem~\ref{res:hil:slices-of-cube-compressed}, they are compressed. Thus Theorem~\ref{thm:relative-ehrhart-hilbert} can be applied to yield the following result.

\begin{theorem}
\label{res:hil:flow-as-hilbert}
For any graph $G$ there exists a relative Stanley-Reisner ideal $I_{\Delta/\Delta'}$ such that for all $k\in\ZZ_{>0}$
\[
\Flow_G(k) = \H_{I_{\Delta/\Delta'}}(k).
\]
\end{theorem}

The above geometric construction can be found in \cite{BeckZaslavsky06b}. A similar construction given in \cite{Dall08} can be used to show an analogue of the above theorem for the integral tension polynomial.

\begin{theorem}
\label{res:hil:tension-as-hilbert}
For any graph $G$ there exists a relative Stanley-Reisner ideal $I_{\Delta/\Delta'}$ such that for all $k\in\ZZ_{>0}$
\[
\Tension_G(k) = \H_{I_{\Delta/\Delta'}}(k).
\]
\end{theorem}


\subsection*{Modular Flow and Tension Polynomials}

Let $v\in V$ and define the vector $a^v\in \{0,\pm1\}^E$ by $a^v_e=+1$ if $\head(e)=v\not=\tail(e)$, $a^v_e=-1$ if $\head(e)\not=v=\tail(e)$, and $a^v_e=0$ otherwise. Let $A$ denote the matrix with the vectors $a^v$ for $v\in V$ as rows. Now, we identify the integers $\{0,\ldots,k-1\}$ with their respective cosets in $\ZZ_k$ so that a function $f:E\rar\ZZ_k$ can be viewed as an integer vector $f\in [0,k)^E$. Using this identification all equations (\ref{eqn:conservation-of-flow}) hold for a given $f$ if and only if $Af=kb$ for some $b\in\ZZ^V$. Thus the set of nowhere zero $\ZZ_k$-flows on $G$ can be identified with the set of lattice points in $k\cdot( (0,1)^E\cap \bigcup_b H_b )$ where $H_b=\{f | Af=b\}$ and $b$ ranges over all integer vectors such that $(0,1)^E\cap H_b\not=\emptyset$. Let $\C$ be the complex generated by the respective closed polytopes $[0,1]^E\cap H_b$. Let $\C'$ be the subcomplex consisting of all those faces that are contained in the boundary of the cube $[0,1]^E$. Then $\mFlow_G(k)=\L_{\bigcup \C\setminus\bigcup \C'}(k)$. The faces of $\C$ are lattice polytopes because $A$ is totally unimodular and by Theorem~\ref{res:hil:slices-of-cube-compressed}, they are compressed. Thus Theorem~\ref{thm:relative-ehrhart-hilbert} can be applied to yield the following result.

\begin{theorem}
\label{res:hil:mflow-as-hilbert}
For any graph $G$ there exists a relative Stanley-Reisner ideal $I_{\Delta/\Delta'}$ such that for all $k\in\ZZ_{>0}$
\[
\mFlow_G(k) = \H_{I_{\Delta/\Delta'}}(k).
\]
\vspace{-3ex}
\end{theorem}

The above geometric construction can be found in \cite{BreuerSanyal09}, which also contains a similar construction using which an analogue of the above theorem for the modular tension polynomial can be shown \cite{Breuer09}.

\begin{theorem}
\label{res:hil:mtension-as-hilbert}
For any graph $G$ there exists a relative Stanley-Reisner ideal $I_{\Delta/\Delta'}$ such that for all $k\in\ZZ_{>0}$
\[
\mTension_G(k) = \H_{I_{\Delta/\Delta'}}(k).
\]
\vspace{-3ex}
\end{theorem}


\subsection*{Chromatic Polynomial}

Let $G$ be a graph without loops.\footnote{If $G$ contains loops, then $\chi_G(k)=0$ which, trivially, is a Hilbert function.} For each $e\in E$ define $H_e=\{x \: | \: x_{\head(e)}=x_{\tail(e)}\}$ and consider the graphic hyperplane arrangement $\HHH=\{H_e \:|\: e\in E\}$. Then the proper $k$-colorings of $G$ are in bijection with the lattice points in $k\cdot ( [0,1)^E \setminus \bigcup \HHH)$. The closure of any component $C$ of $[0,1)^E \setminus \bigcup \HHH$ is of the form $P_\sigma=\{ x\in[0,1]^V \:|\: \sigma_{e}(x_{\head(e)}-x_{\tail(e)})\geq 0\}$ where $\sigma\in\{\pm1\}^E$ is a sign vector. Let $\C$ be the polytopal complex generated by the $P_\sigma$ and let $\C'$ be the subcomplex consisting of all faces that are contained in some hyperplane $H_e$ or in some hyperplane of the form $\{x\:|\:x_v=1\}$ for some $v\in V$. Then $\chi_G(k)=\L_{\bigcup \C\setminus\bigcup \C'}(k)$. The faces of $\C$ are lattice polytopes and it can be shown that they are compressed. Thus Theorem~\ref{thm:relative-ehrhart-hilbert} can be applied to yield the following result.

\begin{theorem}
\label{res:hil:chromatic-as-hilbert}
For any graph $G$ there exists a relative Stanley-Reisner ideal $I_{\Delta/\Delta'}$ such that for all $k\in\ZZ_{>0}$
\[
\chi_G(k) = \H_{I_{\Delta/\Delta'}}(k).
\]
\vspace{-3ex}
\end{theorem}

This is an improvement upon Steingr\'{\i}msson's Theorem insofar as we obtain the chromatic polynomial $\chi_G(k)$ itself as a Hilbert function of a relative Stanley-Reisner ideal and not the shifted polynomial $\chi_G(k+1)$. To obtain the shifted chromatic polynomial using the above construction we would need to consider the closed cube $[0,1]^V$ instead of the half-open cube $[0,1)^V$. A geometric construction similar to the one given above can be found in \cite{BeckZaslavsky06a}.


\section{Non-Square-Free Ideals}
\label{sec:hil:non-square-free-non-standard}

What if the relative polytopal complex $\C$ does not have a unimodular triangulation? It turns out that if the polytopes in $\C$ are normal lattice polytopes, the Ehrhart function of $\bigcup\C\setminus\bigcup\C'$ is still the Hilbert function of a an ideal $I_2$ in a ring $\KK[x]/I_1$, however we cannot guarantee that the ideals $I_1,I_2$ are square-free. That is, we are dealing with a relative multicomplex instead of a relative simplicial complex.

A lattice polytope $P$ is \defn{normal} if for every $k\in\NN$ every $z\in kP\cap\ZZ^d$ can be written as the sum of $k$ points in $P\cap\ZZ^d$. Note that a compressed polytope is automatically normal. 

With this notion we can generalize Theorem~\ref{thm:relative-ehrhart-hilbert} to include another case where the polytopal complex in question satisfies a weaker condition. The conclusion we obtain in this case is not as strong, however.

\begin{theorem}
\label{res:hil:general-theorem}
Let $\C$ be a polytopal complex in which all faces are normal lattice polytopes. Then for any subcomplex $\C'\subset \C$ there exist a monomial ideal $I_1$ in a polynomial ring $\KK[x]$ equipped with the standard grading and a monomial ideal $I_2$ in $\KK[x]/I_1$ such that for all  $k\in\ZZ_{>0}$
\[
\L_{\bigcup\C\setminus\bigcup\C'}(k) = \H_{I_2}(k)
\]
and this function is a polynomial. If the faces of $\C$ are compressed, then moreover the ideals $I_1$ and $I_2$ can be chosen to be square-free.
\end{theorem}

The case where the faces of $\C$ are compressed and the ideals are square-free is just Theorem~\ref{thm:relative-ehrhart-hilbert}. We are not going to prove this again. Instead we give a self-contained algebraic proof of the case where the faces are only normal and we do not conclude that the ideals are square-free.

First we define the \defn{polytopal Stanley-Reisner ideal} $I_\C$ of the polytopal complex $\C$ by
\begin{eqnarray*}
I_\C & :=& \ideal{x^a\;|\;\text{there is no $P\in\C$ such that $\supp(a)\subset P$}}
\end{eqnarray*}
where again $\supp(a)$ denotes the set of lattice points $u$ such that $a_u\not=0$.

\begin{proof}
By homogenization, that is by passing to the complex generated by $\{P\times \{1\}| P\in\C\}$, we can  assume without loss of generality that for every lattice point $z$ there is at most one integer $k$ such that $z\in k\bigcup\C$.

We are going to construct ideals $I_2\supset I_1$ in a polynomial ring $\KK[x]$ such that the monomials in $I_2\setminus I_1$ of degree $k$ are in bijection with the lattice points in $k(\bigcup\C\setminus\bigcup\C')$.

Now consider the polynomial ring $\KK[x_{u}:u\in\bigcup\C\cap\ZZ^d]$ equipped with the standard grading. Let $\lto$ be a term order on this polynomial ring. Let $U$ be the matrix that has the vectors $u$ as columns. Let $n$ be the number of columns of $U$ and let $d$ be the number of rows. We define
\begin{eqnarray*}
I_1 &:=& I_\C + \ideal{x^b \; | \; x^b\not\in I_\C\text{ and there is an } x^a\not\in I_\C \text{ such that } Ua=Ub \text{ and } x^a\lto x^b},\\
I_2 &:= & I_{\C'}
\end{eqnarray*}
where $I_\C$ and $I_{\C'}$ denote the polytopal Stanley-Reisner ideals of the complexes $\C$ and $\C'$ respectively. We call a monomial $x^a$ \emph{valid} if $x^a\not\in I_1$ but $x^a\in I_2$. Now we claim that the map $\pi:x^a \mapsto Ua$ defines a bijection between the valid monomials of degree $k$ and the lattice points in $k(\bigcup\C\setminus\bigcup\C')$.

\emph{If $x^a$ is valid and of degree $k$, then $Ua\in k(\bigcup\C\setminus\bigcup\C')\cap \ZZ^d$.} First, we notice that $Ua\in\ZZ^d$, because $a$ and $U$ are integral. Second, we argue that $Ua \in k\bigcup\C$. Because $x^a$ is valid, there exists a polytope $P$ such that $\supp(a)\subset P$ and thus $Ua\subset kP$. Finally, we show that $Ua \not\in k\bigcup \C'$. Suppose $Ua \in kP'$ for an inclusion-minimal $P'\in\C'$. Because $\C'$ is a subcomplex of $\C$, this implies that $P'$ is a face of $P$ and $\supp(a)\subset P'$. Hence $x^a\not\in I_{\C'}$, which is a contradiction to $x^a$ being valid.

\emph{$\pi$ is surjective.} Let $v\in k(\bigcup\C\setminus\bigcup\C')\cap \ZZ^d$ for some $k$. Then there is a polytope $P\in\C\setminus\C'$ such that $v\in\relint(kP)$. By the assumption that $P$ is normal, there exists a non-negative integral representation $b$ of $v$ in terms of lattice points in $P\cap \ZZ^d$: $v = \sum_{u\in P\cap\ZZ^d} b_u u= Ub$. So by construction $x^b\not\in I_\C$ and $Ub=v$. Consider the $\lto$-minimal monomial $x^a\not\in I_\C$ with $Ua=Ub$. For this monomial we have $x^a\not\in I_1$. Moreover, as $Ua=v\in\relint(kP)$ we have $x^a\in I_2$. Finally we have to check that $\deg(x^a)=k$. All elements of $\supp(a)$ are lattice points in $P$, so $Ua\in\deg(x^a)P$. However by our assumption at the beginning there is at most one integer $k'$ such that $v=Ua\in k'P$. Thus $\deg(x^a)=k$.

\emph{$\hat{\pi}$ is injective.} By definition of $I_1$ and as $\lto$ is a total order on the set of monomials, for every $v\in k(\bigcup\C)\cap \ZZ^d$ there is at most one monomial $x^a\not\in I_1$ such that $Ua=v$.
\end{proof}

We remark that this approach can also be used to give another proof of Theorem~\ref{thm:relative-ehrhart-hilbert} using a fundamental correspondence between compressed polytopes and lattice point sets such that the corresponding toric ideal has square-free initial ideals under any reverse-lexicographic term order (see \cite{Sturmfels96}). This approach is explored in \cite{Dall08} and \cite{Breuer09}. \cite{Breuer09} also contains a variant of the above result, due to Breuer and Sanyal, in the case where $\KK[x]$ is equipped with a non-standard grading.


\section{Bounds on the Coefficients}
\label{sec:questions}

${k-1 \choose d}$ is a polynomial of degree $d$ in $k$. The polynomials ${k-1\choose i}$ for $0\leq i \leq d$ form a basis of the $\KK$-vector space of all polynomials in $\KK[k]$ of degree at most $d$ and the polynomials ${k-1\choose i}$ for $0\leq i$ form a basis of $\KK[k]$ when seen as a $\KK$-vector space. By Theorem~\ref{thm:relative-hilbert-function-and-f-vector} the coefficients of the Hilbert function of a relative Stanley-Reisner ideal expressed with respect to this basis must be non-negative and integral. It turns out that this characterizes which polynomials appear as Hilbert functions of relative Stanley-Reisner ideals.

\begin{theorem}
\label{thm:hil:trivial-bounds}
A polynomial $f(k)=\sum_{i=0}^d f_i {k-1\choose i}$ is the Hilbert function of some relative Stanley-Reisner ideal $I_{\Delta/\Delta'}$ if and only if $f_i\in\ZZ_{\geq0}$ for all $0\leq i\leq d$.
\end{theorem}

\begin{proof}
We have already seen that the coefficients of $\H_{I_{\Delta/\Delta'}}(k)$ with respect to the basis ${k-1\choose d}$, $d\in\ZZ_{\geq 0}$ are necessarily non-negative integers. To see that this is also sufficient, let  $f(k)=\sum_{i=0}^d f_i {k-1\choose i}$ with $f_i\in\ZZ_{\geq0}$ for all $0\leq i\leq d$. For $0\leq i\leq d$ and $1\leq j\leq f_i$ let $\sigma^i_j$ denote a closed unimodular lattice simplex of dimension $i$ in $\RR^d$ such that the $\sigma^i_j$ are pairwise disjoint. Let $\Delta$ denote the (disjoint) union of all these $\sigma^i_j$ and define $\Delta'$ to be the union of the respective boundaries $\partial\sigma^i_j$. Then the set $\bigcup\Delta\setminus\bigcup\Delta'$ is the disjoint union of $f_d$ relatively open unimodular lattice simplices of dimension $d$, $f_{d-1}$ relatively open unimodular lattice simplices of dimension $d-1$ and so on. Consequently $\H_{I_{\Delta/\Delta'}}(k)=\L_{\bigcup\Delta\setminus\Delta'}(k)=f(k)$ as desired.
\end{proof}

This immediately implies that all the counting functions we considered have non-negative integral coefficients with respect to this basis.

\begin{theorem}
\label{res:hil:trivial-bounds-on-polynomials}
The $k$-flow and $\ZZ_k$-flow polynomials, the $k$-tension and $\ZZ_k$-tension polynomials and the chromatic polynomial of a graph have non-negative integer coefficients with respect to the basis $\{{k-1\choose d} | 0\leq d\in\ZZ \}$ of $\KK[k]$.
\end{theorem}

These bounds on the coefficients of the chromatic polynomial are much weaker then the bounds given in \cite{HershSwartz08}, which, because the modular tension polynomial is a divisor of the chromatic polynomial, can also be interpreted as constraints on the modular tension polynomial. We have not been able to find equivalent  or stronger bounds on the coefficients of the modular flow and the integral flow and tension polynomials than those stated in Theorem~\ref{res:hil:trivial-bounds-on-polynomials} in the prior literature. The task of exploiting our results to obtain stronger constraints on these three classes of polynomials is a question for future research. With the integral flow and tension polynomials it seems difficult to obtain strong bounds, as both the complexes $\C$ and $\C'$ have a complex structure in these cases. This is contrary to the case of the chromatic polynomial, where $\C$ is a subdivision of the cube and thus its Ehrhart function can be given explicitly. In the case of the modular flow polynomial, however, there is another realization of $\mFlow_G$ as an Ehrhart function, which can be found in \cite{BreuerSanyal09} and \cite{Breuer09}, such that $\C$ is again a subdivision of the cube. Thus constraints on the Ehrhart polynomial of $\C'$ translate directly into constraints on $\mFlow_G$. We conjecture that in this case, the complex $\C'$ has a convex ear decomposition.

\paragraph{Acknowledgements}
\label{sec:ack}
We thank Matthias Beck and Christian Haase for suggesting this topic to us and for pointing us to the notion of compressed polytopes.

\bibliographystyle{alpha}
\bibliography{article}
\label{sec:biblio}

\end{document}